\documentclass[12pt]{amsart}
\usepackage{amsmath}
\usepackage{amsthm}
\usepackage{tikz}
\usepackage{tikz-cd}
\usepackage{hyperref}
\usepackage[left=2.5cm,right=2.5cm,top=3cm,bottom=3cm]{geometry}
\usepackage{amssymb}
\usepackage{amscd}
\usepackage{latexsym}
\usepackage{epsfig}
\usepackage{graphicx}
\usepackage{psfrag}
\usepackage{caption}
\usepackage{rotating}
\usepackage{mathtools}

\usepackage[normalem]{ulem}

\usepackage{epsfig}

\usepackage{url}
\usepackage{cite}
\usepackage{array}

\newtheorem{theorem}{Theorem}
\newtheorem{proposition}[theorem]{Proposition}
\newtheorem{corollary}[theorem]{Corollary}

\theoremstyle{definition}
\newtheorem{definition}[theorem]{Definition}

\theoremstyle{remark}
\newtheorem*{remark}{Remark}

\DeclareMathOperator{\G}{G}

\DeclareMathOperator{\GL}{GL}

\DeclareMathOperator{\Adj}{Ad}

\DeclareMathOperator{\adj}{ad}

\title[Abstract integrable systems on hyperk\"ahler manifolds]{Abstract integrable systems on hyperk\"ahler manifolds arising from Slodowy slices}
\author{Peter Crooks}
\address{Institute of Differential Geometry, Gottfried Wilhelm Leibniz Universit\"{a}t Hannover, Welfengarten 1, 30167 Hannover, Germany}
\email{~~~peter.crooks@math.uni-hannover.de}

\author{Steven Rayan}
\address{Department of Mathematics \& Statistics, University of Saskatchewan, McLean Hall, Wiggins Road, Saskatoon, SK, S7N 5E6, Canada}
\email{~~~rayan@math.usask.ca}

\begin{document}
	
	\begin{abstract}
		We study holomorphic integrable systems on the hyperk\"ahler manifold $G\times S_{\text{reg}}$, where $G$ is a complex semisimple Lie group and $S_{\text{reg}}$ is the Slodowy slice determined by a regular $\mathfrak{sl}_2(\mathbb{C})$-triple. Our main result is that this manifold carries a canonical \textit{abstract integrable system}, a foliation-theoretic notion recently introduced by Fernandes, Laurent-Gengoux, and Vanhaecke. We also construct traditional integrable systems on $G\times S_{\text{reg}}$, some of which are completely integrable and fundamentally based on Mishchenko and Fomenko's argument shift approach.
	\end{abstract}
	
	\maketitle
	\section{Introduction}
	
	Generalizing the completely integrable systems of mathematical physics are ones described variously as \emph{noncommutative}, \emph{degenerately integrable}, \emph{superintegrable}, or \emph{non-Liouville}.  Such integrable systems have Liouville fibres whose dimensions may be less than half of that of the total space.  They arise in a variety of contexts \cite{Bolsinov,Fiorani,MirandaKiesenhofer,MirandaLaurantVanhaecke,Mishchenko,Reshetikhin}.  In this paper, we adopt the terminology \emph{integrable system of fixed rank}. More precisely, we use the following definition of integrable system in the holomorphic category (cf. \cite[Def. 2.1]{Fernandes}):
	
	\begin{definition}\label{Definition: NCIS}
		Let $X$ be an $n$-dimensional holomorphic symplectic manifold and denote by $\{\cdot,\cdot\}$ the induced Poisson bracket on its sheaf of holomorphic functions. 
		
			 An \textit{integrable system of fixed rank} $r$ on $X$ consists of holomorphic functions$$f_1,f_2,\ldots,f_{n-r}:X\rightarrow\mathbb{C}$$such that:
			\begin{itemize}
				\item[$\bullet$] $r\leq n/2$,
				\item[$\bullet$] $\{f_i,f_j\}=0$ for all $i\in\{1,2,\ldots,r\}$ and $j\in\{1,2,\ldots,n-r\}$, and
				\item[$\bullet$] $f_1,f_2,\ldots,f_{n-r}$ are functionally independent, i.e. the holomorphic $1$-forms \\ $df_1,df_2,\ldots,df_{n-r}$ are linearly independent at each point in an open dense subset of $X$.
			\end{itemize}

	\end{definition}
	
	Note that we recover the usual notion of a completely integrable system by taking $r=n/2$.

	Now, assume that $df_1,df_2,\ldots,df_{n-r}$ are linearly independent at each point of $X$.  It follows that 
	\begin{eqnarray}F\;:\;X & \rightarrow & \mathbb{C}^{n-r}\nonumber\\ x & \mapsto & (f_1(x),f_2(x),\ldots,f_{n-r}(x))\nonumber\end{eqnarray}
	is a holomorphic submersion, whose image is then necessarily open in $\mathbb{C}^{n-r}$. Each fibre is an $r$-dimensional complex submanifold of $X$.  The connected components of all such submanifolds are the leaves of a holomorphic foliation $\mathcal{F}$ of $X$. Note that $f_1,f_2,\ldots,f_{n-r}$ are first integrals of $\mathcal{F}$, meaning that they are constant on the leaves of $\mathcal{F}$, while one can verify that the Hamiltonian vector fields of $f_1,f_2,\ldots,f_{r}$ span $T\mathcal{F}\subseteq TX$ (see \cite[Prop. 2.5]{Fernandes}, for example). These observations motivate a useful definition, introduced by Fernandes, Laurent-Gengoux, and Vanhaecke \cite{Fernandes} to study integrable systems by way of the foliations that they induce. The following is Definition 2.6 from \cite{Fernandes}, adapted to the holomorphic setting and with the term ``noncommutative'' omitted.           
	
	\begin{definition}\label{Definition: ANCIS}
		An \emph{abstract integrable system of rank} $r$ is a pair, $(X,\mathcal{F})$, consisting of a holomorphic symplectic manifold $X$ with an $r$-dimensional holomorphic foliation $\mathcal{F}$, satisfying the following condition: each point $x\in X$ admits an open neighbourhood $U\subseteq X$, together with holomorphic first integrals defined on $U$ whose Hamiltonian vector fields span $T\mathcal{F}$ on $U$.  
	\end{definition} 
	
	The main goal of this work is to construct a hyperk\"ahler variety and an abstract integrable system on it, for which we now give some context and motivation. Recall that a hyperk\"ahler manifold carries a distinguished holomorphic symplectic form, so that one may study integrable systems on hyperk\"ahler manifolds. Moreover, it is from gauge theory that there have emerged deep connections between completely integrable systems and hyperk\"ahler geometry. A particular manifestation of this is the Hitchin system, originating in \cite{Hitchin}. In its typical form, the Hitchin system is an algebraically completely integrable Hamiltonian system defined on the moduli space of stable $G$-Higgs bundles over a fixed algebraic curve or Riemann surface, where $G$ is a reductive complex Lie group.  The moduli space is a noncompact hyperk\"ahler manifold which fits into the picture of Strominger-Yau-Zaslow mirror symmetry through the invariant tori of the dynamical system \cite{HauselThaddeus}.  This proper torus fibration is generally known as the Hitchin map.  

	Recognizing the rich interplay between completely integrable systems and hyperk\"ahler geometry, as well as the recent emergence of abstract integrable systems, we are motivated to think about possible connections between abstract integrable systems and hyperk\"ahler geometry. A natural first step is to construct non-trivial examples of abstract integrable systems carrying hyperk\"ahler structures. It is to the construction of these examples that we devote most of our paper. We also include a few brief digressions on related subjects, such as connections between moment maps and abstract integrable systems, as well as the construction of traditional integrable systems.

	While this helps to explain our motivations, we do believe that the abstract system constructed below is of independent interest as a canonical abstract integrable system associated to purely Lie-theoretic data.

	\subsection{Structure and statement of results} 
	Our paper is organized as follows: Section \ref{Section: Background} introduces the manifold$$M\;:=\;G\times S_\text{reg},$$ where $G$ is a connected, simply-connected, complex semisimple linear algebraic group with Lie algebra $\mathfrak{g}$, and $S_\text{reg}\subseteq\mathfrak{g}$ is the Slodowy slice determined by a regular $\mathfrak{sl}_2(\mathbb C)$-triple (defined formally in Section \ref{SubsectionSlodowy}). Following Bielawski in \cite{Bielawski}, $M$ has a canonical hyperk\"ahler structure (see Theorem \ref{Theorem: Bielawski's theorem}).
	
	In Section \ref{Section: Main results}, we consider the holomorphic map
	\begin{eqnarray}\Phi\;:\;M & \rightarrow & \mathfrak{g}\nonumber\\ (g,x) & \mapsto & -\Adj_{g^{-1}}(x),\nonumber\end{eqnarray} where $\Adj:G\rightarrow\GL(\mathfrak{g})$ is the adjoint representation. We show that $\Phi$ is a holomorphic submersion with image the set of regular elements $\mathfrak g_\text{reg}\subseteq\mathfrak{g}$, and that each fibre $\Phi^{-1}(x)\subseteq M$ is an isotropic subvariety of dimension $\mbox{rank}(G)$ that is isomorphic to the $G$-stabilizer of $x$.  These results are contained in Proposition \ref{Proposition: Surjective submersion}, Corollary \ref{Corollary: Connected components}, and Proposition \ref{Proposition: Isotropic fibres}.  Furthermore, $\Phi$ is a map of Poisson varieties for the holomorphic symplectic structure on $M$ and the Poisson structure on $\mathfrak g$ coming from the Killing form-induced isomorphism $\mathfrak{g}\cong\mathfrak{g}^*$.  This is Proposition \ref{Proposition: Poisson morphism}.  Next, we foliate $M$ into the connected components of $\Phi$'s fibres, and argue in Theorem \ref{Theorem: Phi is ANCI} that this constitutes an abstract integrable system of rank equal to $\text{rank}(G)$. Section \ref{Section: Main results} concludes with some consideration of Theorem \ref{Theorem: Phi is ANCI} in a more general context. We give conditions under which a moment map will, analogously to $\Phi$ in Theorem \ref{Theorem: Phi is ANCI}, induce an abstract integrable system. These results are contained in Theorem \ref{Theorem: General symplectic approach}. 
	
	Section \ref{Section: Some integrable systems} applies the results of Section \ref{Section: Main results} to construct traditional integrable systems on $G\times S_{\text{reg}}$. The first such system appears in Section \ref{Subsection: An integrable system of rank equal to rk(G)}, and is an integrable system of rank equal to $\text{rank}(G)$ (see Theorem \ref{Theorem: An integrable system}). In Section \ref{Subsection: A family of completely integrable systems}, we use Mishchenko and Fomenko's argument shift approach \cite{MishchenkoEuler} to construct a family of completely integrable systems on $G\times S_{\text{reg}}$. This results in Theorem \ref{Theorem: Family of completely integrable systems}.

	Before proceeding with our construction, we make a few informal remarks regarding analogies with the Hitchin system: unlike the Hitchin fibration, which is necessarily proper, the fibration given by $\Phi$ is one whose generic fibres are noncompact complex tori (see Section \ref{Subsection: Structure of the fibres}).  Interestingly, this mirrors certain Hitchin-type examples in Section 3 of \cite{Vanhaecke}. Moreover, our abstract integrable system is akin to the Hitchin-type systems studied by Beauville, Markman, Biswas-Ramanan, and Bottacin in \cite{Beauville,Markman,BiswasRamanan,Bottacin}, respectively, in that the dimension of the base is allowed to exceed that of the fibres.  We also recognize that there is a well-known passage between Hitchin systems and Mishchenko-Fomenko flows: \cite{AdamsHarnadHurtubise} makes a connection between Adler-Kostant-Symes flows and the flows of \cite{MishchenkoEuler}, while \cite{DonagiMarkman} explains the route from the AKS picture to Hitchin systems.

	\subsection*{Acknowledgements}  We thank Roger Bielawski, Hartmut Wei\ss, and Jonathan Weitsman for useful feedback at several points during the writing of this manuscript.  The first author is grateful to Stefan Rosemann and Markus R\"oser for several interesting discussions while the second author likewise acknowledges Jacek Szmigielski.  The first author was supported by a postdoctoral fellowship at the Institute of Differential Geometry, Leibniz Universit\"{a}t Hannover. The second author was supported by a University of Saskatchewan New Faculty Research Grant.\\

	\section{Background}\label{Section: Background}
	\subsection{Lie-theoretic preliminaries}\label{Subsection: Some general Lie-theoretic preliminaries}
	Let $G$ be a connected, simply-connected, complex semisimple linear algebraic group and denote its rank by $\text{rk}(G)$. We shall use $\mathfrak{g}$ to denote the Lie algebra of $G$, on which one has a Killing form $\langle\cdot,\cdot\rangle:\mathfrak{g}\otimes\mathfrak{g}\rightarrow\mathbb{C}$ and exponential map $\exp:\mathfrak{g}\rightarrow G$.
	One also has the adjoint and coadjoint representations of $G$, denoted \begin{eqnarray}\Adj\;:\;G & \rightarrow & \GL(\mathfrak{g})\nonumber\\ g & \mapsto & \Adj_g\nonumber\end{eqnarray}and\begin{eqnarray}\Adj^*\;:\;G & \rightarrow & \GL(\mathfrak{g^*})\nonumber\\g & \mapsto & \Adj^*_g,\nonumber\end{eqnarray} respectively. Since the Killing form is non-degenerate and invariant under the first representation, the following is an isomorphism between the adjoint and coadjoint representations:
	\begin{equation}\label{Equation: Isomorphism of representations}
	\mathfrak{g}\xrightarrow{\cong}\mathfrak{g}^*,\quad x\mapsto x^{\vee}:=\langle x,\cdot\rangle.
	\end{equation}
	The canonical Lie-Poisson structure on $\mathfrak{g}^*$ (see \cite[Prop. 1.3.18]{Chriss}) thereby corresponds to a holomorphic Poisson structure on $\mathfrak{g}$, whose symplectic leaves turn out to be the adjoint orbits of $G$. We shall let $\mathcal{O}(x)$ shall denote the adjoint orbit containing $x\in\mathfrak{g}$, i.e.
	$$\mathcal{O}(x):=\{\Adj_g(x):g\in G\}\subseteq\mathfrak{g}.$$ 
	
	Now let \begin{eqnarray}\adj\;:\;\mathfrak{g} & \rightarrow & \mathfrak{gl}(\mathfrak{g})\nonumber\\ x & \mapsto & \adj_x\nonumber\end{eqnarray} denote the adjoint representation of $\mathfrak{g}$. Recall that $\adj$ is the derivative of $\Adj$ at the identity $e\in G$ and satisfies $\adj_x(y)=[x,y]$ for all $x,y\in\mathfrak{g}$. Furthermore, recall that $x\in\mathfrak{g}$ is called \textit{semisimple} (resp. \textit{nilpotent}) if the endomorphism $\adj_x:\mathfrak{g}\rightarrow\mathfrak{g}$ is semisimple (resp. nilpotent). Denote by $\mathfrak{g}_{\text{ss}}$ and $\mathcal{N}$ the sets of semisimple and nilpotent elements in $\mathfrak{g}$, respectively. The former is an open dense subvariety of $\mathfrak{g}$ while the latter is a closed subvariety called the \textit{nilpotent cone}. Each is readily seen to be invariant under the adjoint representation of $G$, and one calls an adjoint orbit $\mathcal{O}\subseteq\mathfrak{g}$ \textit{semisimple} (resp. \textit{nilpotent}) if $\mathcal{O}\subseteq\mathfrak{g}_{\text{ss}}$ (resp. $\mathcal{O}\subseteq\mathcal{N}$). 
	
	We shall use $Z_G(x)$ to denote the $G$-stabilizer of $x\in\mathfrak{g}$, i.e. $$Z_G(x):=\{g\in G:\Adj_g(x)=x\}.$$ Its Lie algebra is the $\mathfrak{g}$-stabilizer of $x$, namely $$Z_{\mathfrak{g}}(x):=\{y\in\mathfrak{g}:[x,y]=0\}.$$ Now recall that the dimension of $Z_{\mathfrak{g}}(x)$ (equivalently, the dimension of $Z_G(x)$) is at least $\text{rk}(G)$ for all $x\in\mathfrak{g}$, and that $x$ is called \textit{regular} when $\dim(Z_{\mathfrak{g}}(x))=\text{rk}(G)$. Note that while some authors also require regular elements to be semisimple, we are not imposing this extra condition. 
	
	Let us consider the set of regular elements,
	$$\mathfrak{g}_{\text{reg}}:=\{x\in\mathfrak{g}:x\text{ is regular}\},$$ which is known to be a $G$-invariant open dense subset of $\mathfrak{g}$ (see \cite[Lemma 6.51]{Kirillov}). We shall call an adjoint orbit $\mathcal{O}\subseteq\mathfrak{g}$ \textit{regular} if $\mathcal{O}\subseteq\mathfrak{g}_{\text{reg}}$. Note that while there exist infinitely many distinct regular semisimple orbits, there exists exactly one regular nilpotent orbit, to be denoted $\mathcal{O}_{\text{reg}}$ (see \cite[Prop. 3.2.10]{Chriss}, for example).
	
	\subsection{Regular $\mathfrak{sl}_2(\mathbb{C})$-triples and Slodowy slices}\label{SubsectionSlodowy}
	Given three vectors $\xi,h,\eta\in\mathfrak{g}$, recall that $(\xi,h,\eta)$ is called an $\mathfrak{sl}_2(\mathbb{C})$-triple if $[\xi,\eta]=h$, $[h,\xi]=2\xi$, and $[h,\eta]=-2\eta$. The elements $\xi$ and $\eta$ are then necessarily nilpotent and belong to the same nilpotent orbit. With this in mind, we shall call our $\mathfrak{sl}_2(\mathbb{C})$-triple $(\xi,h,\eta)$ \textit{regular} if $\xi,\eta\in\mathcal{O}_{\text{reg}}$. We will often restrict our attention to triples of this form.
	
	Now fix an $\mathfrak{sl}_2(\mathbb{C})$-triple $(\xi,h,\eta)$. One can associate a \textit{Slodowy slice} to this triple, namely the affine-linear subspace $S(\xi,h,\eta)\subseteq\mathfrak{g}$ defined as follows: 
	$$S(\xi,h,\eta):=\xi+Z_{\mathfrak{g}}(\eta):=\{\xi+x:x\in Z_{\mathfrak{g}}(\eta)\}.$$ Named in recognition of Slodowy's work \cite{Slodowy}, this variety is well-studied and enjoys a number of remarkable properties. To formulate one of these properties, recall that two smooth subvarieties $Y,Z\subseteq\mathfrak{g}$ are called \textit{transverse} if for all $x\in Y\cap Z$, we have $T_xY+T_xZ=\mathfrak{g}$ as vector spaces. It turns out that if $(\xi,h,\eta)$ is an $\mathfrak{sl}_2(\mathbb{C})$-triple, then $S(\xi,h,\eta)$ and $\mathcal{O}(\xi)$ are transverse and intersect only at $\xi$.       
	
	It is desirable to have a more precise understanding of orbit-slice intersections when the underlying $\mathfrak{sl}_2(\mathbb{C})$-triple is regular. The following are parts of Theorem 8 and Lemma 13 from Kostant's paper \cite{KostantLie}, rephrased to suit our purposes.
	
	\begin{theorem}[Kostant]\label{Theorem: Kostant's theorem}
		If $(\xi,h,\eta)$ is a regular $\mathfrak{sl}_2(\mathbb{C})$-triple, then $S(\xi,h,\eta)\subseteq\mathfrak{g}_{\emph{reg}}$ and
		each regular adjoint orbit intersects $S(\xi,h,\eta)$ in a single point. Moreover, $S(\xi,h,\eta)$ is transverse to each regular adjoint orbit.  
	\end{theorem}    
	
	\subsection{The hyperk\"{a}hler varieties of interest}\label{Subsection: The hyperkahler varieties of interest}
	Recall that the cotangent bundle $T^*G$ carries a distinguished holomorphic symplectic form, to be denoted $\omega_{T^*G}$. It will be convenient to identify $T^*G$ with $G\times\mathfrak{g}$, using an isomorphism between the two to transport the holomorphic symplectic structure on the former to one on the latter. We define this isomorphism as follows: 
	
	\begin{eqnarray}\label{Equation: Second trivialization}\Psi\;:\;G\times\mathfrak{g} & \xrightarrow{\cong} & T^*G\\ (g,x) & \mapsto & (g,x^{\vee}\circ (d_eR_g)^{-1}),\nonumber\end{eqnarray}
	where $R_g:G\rightarrow G$ denotes right multiplication by a fixed $g\in G$ and $d_eR_g:\mathfrak{g}\rightarrow T_gG$ is the differential of $R_g$ at $e\in G$. Note that \eqref{Equation: Second trivialization} is simply the result of composing the right trivialization $G\times\mathfrak{g}^*\rightarrow T^*G$ with the isomorphism $G\times\mathfrak{g}\rightarrow G\times\mathfrak{g}^*$, $(g,x)\mapsto (g,x^{\vee})$. 
	
	Let $\omega:=\Psi^*(\omega_{T^*G})$ denote the induced holomorphic symplectic form on $G\times\mathfrak{g}$. Given $(g,x)\in G\times\mathfrak{g}$, one can verify that $\omega$ restricts to the following bilinear form on $T_{(g,x)}(G\times\mathfrak{g})=T_gG\oplus\mathfrak{g}$:
	\begin{equation}\label{Equation: Symplectic form in right trivialization}
	\omega_{(g,x)}\big((d_eR_g(y_1),z_1), (d_eR_g(y_2),z_2)\big)=\langle y_1,z_2\rangle-\langle y_2,z_1\rangle-\langle x,[y_1,y_2]\rangle,
	\end{equation}
	where $y_1,y_2,z_1,z_2\in\mathfrak{g}$ (cf. \cite[Sect. 5, Eqn. (14R)]{Marsden}).
	
	Now let $(\xi,h,\eta)$ be an $\mathfrak{sl}_2(\mathbb{C})$-triple with associated Slodowy slice $S(\xi,h,\eta)\subseteq\mathfrak{g}$. One has an inclusion of varieties $G\times S(\xi,h,\eta)\subseteq G\times\mathfrak{g}$, by virtue of which the former carries some interesting geometric structures. Indeed, the following is one of Bielawski's results (see \cite{Bielawski}), phrased in terms of $G\times\mathfrak{g}$ instead of $T^*G$
	
	\begin{theorem}[Bielawski]\label{Theorem: Bielawski's theorem}
		If $(\xi,h,\eta)$ is an $\mathfrak{sl}_2(\mathbb{C})$-triple, then $G\times S(\xi,h,\eta)$ has a canonical hyperk\"{a}hler structure whose underlying holomorphic symplectic form is obtained by restricting that of $G\times\mathfrak{g}$ to $G\times S(\xi,h,\eta)$.
	\end{theorem}     
	
	Now let $G$ act on $G\times S(\xi,h,\eta)$ according to
	\begin{equation}\label{Equation: Hamiltonian action} h\cdot (g,x):=(gh^{-1},x),\quad g,h\in G,\text{ } x\in S(\xi,h,\eta).\end{equation} This action enjoys a few properties that we record here for future reference.
	
	\begin{proposition}\label{Proposition: Hamiltonian action}
		The action \eqref{Equation: Hamiltonian action} is Hamiltonian with respect to the holomorphic symplectic form on $G\times S(\xi,h,\eta)$, and
		\begin{eqnarray}\label{Equation: Moment map}\mu\;:\;G\times S(\xi,h,\eta) & \rightarrow & \mathfrak{g}^*\\ (g,x) & \mapsto & -\Adj_{g^{-1}}^*(x^{\vee})\nonumber\end{eqnarray}
		is a moment map.
	\end{proposition}
	
	\begin{proof}
		Consider the following extension of \eqref{Equation: Hamiltonian action} to an action on $G\times\mathfrak{g}$:
		\begin{equation}\label{Equation: Extended Hamiltonian action} h\cdot (g,x):=(gh^{-1},x),\quad g,h\in G,\text{ } x\in\mathfrak{g}.\end{equation}
		We will show \eqref{Equation: Extended Hamiltonian action} to be a Hamiltonian action with respect to the symplectic form \eqref{Equation: Symplectic form in right trivialization}, and that 
		\begin{eqnarray}\tilde{\mu}\;:\;G\times\mathfrak{g} & \rightarrow & \mathfrak{g}^*\nonumber\\ (g,x) & \mapsto & -\Adj_{g^{-1}}^*(x^{\vee})\nonumber\end{eqnarray}
		is a moment map. Since $G\times S(\xi,h,\eta)$ is a $G$-invariant symplectic subvariety of $G\times\mathfrak{g}$, it will follow that \eqref{Equation: Hamiltonian action} is a Hamiltonian action with moment map $\tilde{\mu}\vert_{G\times S(\xi,h,\eta)}=\mu$.
		
		Now consider the following action of $G$ on $T^*G$:
		\begin{equation}\label{Equation: Induced action on cotangent bundle}h\cdot(g,\alpha):=(gh^{-1},\alpha\circ(d_gR_{h^{-1}})^{-1}),\quad g,h\in G,\text{ }\alpha\in T^*_gG.
		\end{equation}
		This is the action on $T^*G$ naturally induced by an action of $G$ on itself, namely \begin{equation}\label{Equation: Action on G} h\cdot g:=gh^{-1},\quad g,h\in G.\end{equation} As such, \eqref{Equation: Induced action on cotangent bundle} is necessarily a Hamiltonian action with the following moment map: $$\Theta:T^*G\rightarrow\mathfrak{g}^*,\quad \Theta(g,\alpha)(y)=\alpha(X_{y}(g)),\quad g\in G,\text{ }\alpha\in T^*_gG,\text{ }y\in\mathfrak{g}$$
		(see \cite[Example 4.5.4]{Ortega}), where $X_y$ is the fundamental vector field on $G$ for the action \eqref{Equation: Action on G} and the point $y\in\mathfrak{g}$. One can verify that $X_y(g)=-d_eL_g(y)\in T_gG$ for all $y\in\mathfrak{g}$ and $g\in G$.
		
		Recall the isomorphism $\Psi:G\times\mathfrak{g}\rightarrow T^*G$ from \eqref{Equation: Second trivialization}. It is not difficult to verify that
		$\Psi$ is $G$-equivariant with respect to the actions \eqref{Equation: Extended Hamiltonian action} and \eqref{Equation: Induced action on cotangent bundle}. Moreover, as $\omega=\Psi^*(\omega_{T^*G})$, it follows that $\Psi$ is a $G$-equivariant symplectomorphism. Since \eqref{Equation: Induced action on cotangent bundle} is a Hamiltonian action with moment map $\Theta$, this implies that \eqref{Equation: Extended Hamiltonian action} is also a Hamiltonian action with moment map $\Theta\circ\Psi$. It then remains only to prove that $\tilde{\mu}=\Theta\circ\Psi$.
		
		Given $g\in G$ and $x,y\in\mathfrak{g}$, we have
		\begin{align*}(\Theta\circ\Psi)(g,x)(y) & = \Theta(g,x^{\vee}\circ (d_eR_g)^{-1})(y)\\
		& = (x^{\vee}\circ (d_eR_g)^{-1})(X_y(g))\\
		& = \langle x,(d_eR_g)^{-1}(X_y(g))\rangle\\
		& = \langle x,-(d_eR_g)^{-1}(d_eL_g(y))\rangle\\
		& = \langle x,-\Adj_g(y)\rangle\\
		& = \langle -\Adj_{g^{-1}}(x),y\rangle\\
		& = (-\Adj_{g^{-1}}(x))^{\vee}(y).
		\end{align*}
		It follows that $(\Theta\circ\Psi)(g,x)=(-\Adj_{g^{-1}}(x))^{\vee}$, and in turn the right-hand side becomes$$-\Adj_{g^{-1}}^*(x^{\vee})=\tilde{\mu}(g,x)$$ when one recalls that \eqref{Equation: Isomorphism of representations} intertwines the adjoint and coadjoint representations. 
	\end{proof}     
	
	\section{A canonical abstract integrable system}\label{Section: Main results}
	
	\subsection{The map $G\times S_{\text{reg}}\rightarrow\mathfrak{g}$ and its properties}\label{Subsection: The map}
	For the duration of this article, $(\xi,h,\eta)$ will be a fixed regular $\mathfrak{sl}_2(\mathbb{C})$-triple and $S_{\text{reg}}:=S(\xi,h,\eta)$ shall denote its associated Slodowy slice. 
	One may then consider the holomorphic map
	\begin{eqnarray}\label{Equation: Definition of phi}\Phi\;:\;G\times S_{\text{reg}} & \rightarrow & \mathfrak{g}\\ (g,x) & \mapsto & -\Adj_{g^{-1}}(x).\nonumber\end{eqnarray}
	
	A preliminary observation is that $\Phi$ is $G$-equivariant for the adjoint action on $\mathfrak{g}$ and the following $G$-action on $G\times S_{\text{reg}}$:
	\begin{equation}\label{Equation: Action on hyperkahler variety}h\cdot(g,x):=(gh^{-1},x),\quad g,h\in G,\text{  }x\in S_{\text{reg}}.
	\end{equation} 
	However, one can say considerably more about $\Phi$.
	\begin{proposition}\label{Proposition: Surjective submersion}
		The map $\Phi$ is a holomorphic submersion and its image is $\mathfrak{g}_{\emph{reg}}$.
	\end{proposition}
	
	\begin{proof}
		Theorem \ref{Theorem: Kostant's theorem} gives the inclusion $S_{\text{reg}}\subseteq\mathfrak{g}_{\text{reg}}$, and the latter set is invariant under both the adjoint $G$-action and multiplication by $-1$. From this last sentence, it follows that $-\Adj_{g^{-1}}(x)\in\mathfrak{g}_{\text{reg}}$ for all $g\in G$ and $x\in S_{\text{reg}}$, i.e. $\Phi(G\times S_{\text{reg}})\subseteq \mathfrak{g}_{\text{reg}}$. For the opposite inclusion, suppose that $x\in\mathfrak{g}_{\text{reg}}$. It follows that $-x$ belongs to a regular adjoint orbit, which by Theorem \ref{Theorem: Kostant's theorem} must intersect $S_{\text{reg}}$ at a point $y$. Note that $-x=\Adj_{g^{-1}}(y)$ for some $g\in G$, so that $\Phi(g,y)=x$. We conclude that $\mathfrak{g}_{\text{reg}}\subseteq\Phi(G\times S_{\text{reg}})$, completing our proof that $\mathfrak{g}_{\text{reg}}$ is the image of $\Phi$.  
		
		To show that $\Phi$ is submersive is to show that the differential of $\Phi$ at $(g,x)$, $d_{(g,x)}\Phi$, is a surjective map of tangent spaces for all $(g,x)\in G\times S_{\text{reg}}$. However, since $\Phi$ is $G$-equivariant in the sense discussed before this proposition, it will suffice prove that $$d_{(e,x)}\Phi:T_{(e,x)}(G\times S_{\text{reg}})\rightarrow T_{-x}(\mathfrak{g})=\mathfrak{g}$$ is surjective for all $x\in S_{\text{reg}}$. To this end, note that $T_{(e,x)}(G\times S_{\text{reg}})$ is canonically identified with $\mathfrak{g}\oplus Z_{\mathfrak{g}}(\eta)$. Also, given $(y,z)\in\mathfrak{g}\oplus Z_{\mathfrak{g}}(\eta)$, observe that $t\mapsto(\exp(ty),x+tz)$ is a curve in $G\times S_{\text{reg}}$ having tangent vector $(y,z)$ at $t=0$. Using the previous two statements, we may present $d_{(e,x)}\Phi$ as a map
		$$d_{(e,x)}\Phi:\mathfrak{g}\oplus Z_{\mathfrak{g}}(\eta)\rightarrow\mathfrak{g}$$ whose value at the tangent vector $(y, z)\in\mathfrak{g}\oplus Z_{\mathfrak{g}}(\eta)$ is calculated as follows:
		\begin{align*}
		(d_{(e,x)}\Phi)(y,z) & =\frac{d}{dt}\bigg\vert_{t=0}\bigg(\Phi(\exp(ty),x+tz)\bigg)\\
		&=\frac{d}{dt}\bigg\vert_{t=0}\bigg(-\Adj_{\exp(-ty)}(x+tz)\bigg)\\
		&=\frac{d}{dt}\bigg\vert_{t=0}\bigg(-\exp(\adj_{-ty})(x+tz)\bigg)\\
		&=[y,x]-z.
		\end{align*}
		Noting that $T_x\mathcal{O}(x)=\{[y,x]:y\in\mathfrak{g}\}$, this calculation shows the image of $d_{(e,x)}\Phi$ to be precisely $T_x\mathcal{O}(x)+Z_{\mathfrak{g}}(\eta)$. Since $\mathcal{O}(x)$ and $S_{\text{reg}}$ are transverse (by Theorem \ref{Theorem: Kostant's theorem}), this image is necessarily all of $\mathfrak{g}$. We conclude that $d_{(e,x)}\Phi$ is surjective for all $x\in S_{\text{reg}}$, as required.
	\end{proof}
	
	For future reference, we record the following immediate consequence of Proposition \ref{Proposition: Surjective submersion}.
	
	\begin{corollary}\label{Corollary: Connected components}
		If $x\in\mathfrak{g}_{\emph{reg}}$, then $\Phi^{-1}(x)$ and its connected components are complex submanifolds of $G\times S_{\emph{reg}}$ having dimension $\emph{rk}(G)$.
	\end{corollary}
	
	\begin{proof}
		Proposition \ref{Proposition: Surjective submersion} implies that $\Phi^{-1}(x)$ and its connected components are complex submanifolds of dimension $\dim(G)+\dim(S_{\text{reg}})-\dim(\mathfrak{g})=\dim(S_{\text{reg}})$. Furthermore, $\dim(S_{\text{reg}})=\dim(S(\xi,h,\eta))=\dim(Z_{\mathfrak{g}}(\eta))=\text{rk}(G)$, with the third equality following from the fact that $\eta$ is regular. This completes the proof.
	\end{proof}
	
	It turns out that $\Phi$ enjoys some additional structure. To describe it, recall from Section \ref{Subsection: Some general Lie-theoretic preliminaries} that $\mathfrak{g}$ is canonically a holomorphic Poisson variety. At the same time, $G\times S_{\text{reg}}$ is Poisson by virtue of inheriting a holomorphic symplectic form from $G\times\mathfrak{g}$ (see Theorem \ref{Theorem: Bielawski's theorem}). These considerations give context for the following result. 
	
	\begin{proposition}\label{Proposition: Poisson morphism}
		The map $\Phi$ is a morphism of Poisson varieties.
	\end{proposition}
	
	\begin{proof}
		Since $\mathfrak{g}$ inherits its Poisson structure from $\mathfrak{g}^*$ and the isomorphism \eqref{Equation: Isomorphism of representations}, it suffices to show that the composition of $\Phi$ with \eqref{Equation: Isomorphism of representations} is Poisson. However, this composite map coincides with the moment map $\mu$ from \eqref{Equation: Moment map}, and (equivariant) moment maps for Hamiltonian $G$-actions are necessarily Poisson (see \cite[Lem. 1.4.2(ii)]{Chriss}). This completes the proof. 
	\end{proof}
	
	\subsection{Structure of the fibres}\label{Subsection: Structure of the fibres} Let us take a moment to examine the nonempty fibres of $\Phi$, which by Proposition \ref{Proposition: Surjective submersion} are precisely those fibres of the form $\Phi^{-1}(x)$, $x\in\mathfrak{g}_{\text{reg}}$. For each such element $x$, Theorem \ref{Theorem: Kostant's theorem} implies that $S_{\text{reg}}$ intersects $\mathcal{O}(-x)$ in a single point, $\tilde{x}\in\mathcal{O}(-x)\cap S_{\text{reg}}$. Since $\tilde{x}$ belongs to the orbit of $-x$, there exists a (non-unique) $g\in G$ such that $\Adj_{g^{-1}}(\tilde{x})=-x$. In what follows, we show $\Phi^{-1}(x)$ to be $R_g(Z_G(\tilde{x}))\times\{\tilde{x}\}\subseteq G\times S_{\text{reg}}$.
	
	\begin{proposition}\label{Proposition: Fibre description}
		If $x\in\mathfrak{g}_{\emph{reg}}$, then 
		\begin{equation}\label{Equation: Fibre description}\Phi^{-1}(x)=R_g(Z_G(\tilde{x}))\times\{\tilde{x}\},
		\end{equation}
		where $\tilde{x}$ is the unique element in $\mathcal{O}(-x)\cap S_{\emph{reg}}$ and $g\in G$ is any element satisfying $\Adj_{g^{-1}}(\tilde{x})=-x$.
	\end{proposition}    
	
	\begin{proof}
		To see the inclusion ``$\supseteq$'',  one directly verifies $\Phi(hg,\tilde{x})=x$ for all $h\in Z_G(\tilde{x})$. This is a straightforward exercise. As for the other inclusion, suppose $(h,y)\in\G\times S_{\text{reg}}$ satisfies $\Phi(h,y)=x$, i.e. $-\Adj_{h^{-1}}(y)=x$. It follows that $y=\Adj_h(-x)$ belongs to $\mathcal{O}(-x)$, which together with the fact that $y\in S_{\text{reg}}$ implies $y=\tilde{x}$. Furthermore,
		$$\Adj_{h^{-1}}(\tilde{x})=\Adj_{h^{-1}}(y)=-x=\Adj_{g^{-1}}(\tilde{x}),$$ which one can manipulate to show $\Adj_{hg^{-1}}(\tilde{x})=\tilde{x}$. We conclude that $hg^{-1}\in Z_G(\tilde{x})$, so that $h\in R_g(Z_G(\tilde{x}))$. Hence $(h,y)\in R_g(Z_G(\tilde{x}))\times\{\tilde{x}\}$, as desired.
	\end{proof}
	
	\begin{proposition}\label{Proposition: Isotropic fibres}
		If $x\in\mathfrak{g}_{\emph{reg}}$, then $\Phi^{-1}(x)$ is an isotropic subvariety of $G\times S_{\emph{reg}}$.
	\end{proposition}
	
	\begin{proof}
		Let $\omega$ denote the holomorphic symplectic form on $G\times\mathfrak{g}$, as described in \ref{Subsection: The hyperkahler varieties of interest}. Since the holomorphic symplectic form on $G\times S_{\text{reg}}$ is obtained by restricting $\omega$ (see Theorem \ref{Theorem: Bielawski's theorem}), proving the proposition amounts to showing that $\omega$ restricts to the zero-form on tangent spaces of $\Phi^{-1}(x)$. To identify these tangent spaces, let $\tilde{x}\in\mathcal{O}(-x)\cap S_{\text{reg}}$ and $g\in G$ be as in the statement of Proposition \ref{Proposition: Fibre description}. The proposition implies that each point in $\Phi^{-1}(x)$ has the form $(hg,\tilde{x})$ for $h\in Z_G(\tilde{x})$, and that the tangent space of $\Phi^{-1}(x)$ at $(hg,\tilde{x})$ is the following subspace of $T_{(hg,\tilde{x})}(G\times\mathfrak{g})=T_{hg}G\oplus\mathfrak{g}$:
		\begin{equation}\label{Equation: First tangent space description} T_{(hg,\tilde{x})}(\Phi^{-1}(x))=T_{hg}(R_g(Z_G(\tilde{x})))\oplus\{0\}\subseteq T_{hg}G\oplus\mathfrak{g}.\end{equation}
		Now note that $T_{hg}(R_g(Z_G(\tilde{x})))$ is the right $g$-translate of the tangent space to $Z_G(\tilde{x})$ at $h$, i.e. $T_{hg}(R_g(Z_G(\tilde{x})))=d_hR_g(T_hZ_G(\tilde{x}))$. At the same time, $T_hZ_G(\tilde{x})$ is the right $h$-translate of the tangent space to $Z_G(\tilde{x})$ at $e$, meaning $T_hZ_G(x)=d_eR_h(Z_{\mathfrak{g}}(\tilde{x}))$. It follows that
		$$T_{hg}(R_g(Z_G(\tilde{x})))=d_hR_g(T_hZ_G(\tilde{x}))=d_hR_g(d_eR_h(Z_{\mathfrak{g}}(\tilde{x})))=d_eR_{hg}(Z_{\mathfrak{g}}(\tilde{x})).$$ In particular, one may rewrite \eqref{Equation: First tangent space description} as the statement
		$$T_{(hg,\tilde{x})}(\Phi^{-1}(x))=d_eR_{hg}(Z_{\mathfrak{g}}(\tilde{x}))\oplus\{0\}\subseteq T_{hg}G\oplus\mathfrak{g}.$$ We are therefore reduced to verifying that
		$$\omega_{(hg,\tilde{x})}\big((d_eR_{hg}(y_1),0), (d_eR_{hg}(y_2),0)\big)=0$$ for all $y_1,y_2\in Z_{\mathfrak{g}}(\tilde{x})$. To this end, \eqref{Equation: Symplectic form in right trivialization} implies that $$\omega_{(hg,\tilde{x})}\big((d_eR_{hg}(y_1),0), (d_eR_{hg}(y_2), 0)\big)=-\langle\tilde{x},[y_1,y_2]\rangle.$$
		Using the Killing form's $\adj$-invariance property, the right-hand-side becomes $-\langle[\tilde{x},y_1],y_2\rangle$. This is necessarily zero, as $y_1\in Z_{\mathfrak{g}}(\tilde{x})$.
	\end{proof}
	
	We conclude this section by characterizing various fibres of $\Phi$ up to variety isomorphism. While not strictly essential to proving our main results, these characterizations are in keeping with the long-standing interest in understanding generic and non-generic fibres of integrable systems. We begin with the following proposition. 
	
	\begin{proposition}\label{Proposition: Fibre isomorphism}
		If $x\in\mathfrak{g}_{\emph{reg}}$, then $\Phi^{-1}(x)$ and $Z_G(x)$ are isomorphic as varieties.
	\end{proposition}
	
	\begin{proof}
		Let $\tilde{x}$ be as introduced in Proposition \ref{Proposition: Fibre description}, so that \eqref{Equation: Fibre description} implies $\Phi^{-1}(x)$ and $Z_G(\tilde{x})$ are isomorphic as varieties. It then remains to prove that $Z_G(\tilde{x})$ and $Z_G(x)$ are isomorphic. Now, since $\tilde{x}\in\mathcal{O}(-x)$, one sees that $Z_G(\tilde{x})$ and $Z_G(-x)$ are conjugate in $G$. The latter stabilizer coincides with $Z_G(x)$, so that $Z_G(\tilde{x})$ and $Z_G(x)$ are conjugate in $G$. In particular, $Z_G(\tilde{x})$ and $Z_G(x)$ are isomorphic as varieties (in fact, as algebraic groups).
	\end{proof} 
	
	Now note that $Z_G(x)$ is a maximal torus of $G$ whenever $x\in\mathfrak{g}_{\text{reg}}\cap\mathfrak{g}_{\text{ss}}$ (cf. \cite[Lem. 2.1.9, Thm. 2.3.3]{Collingwood}). Proposition \ref{Proposition: Fibre isomorphism} then implies that fibres of $\Phi$ over $\mathfrak{g}_{\text{reg}}\cap\mathfrak{g}_{\text{ss}}$ are isomorphic to maximal tori of $G$, or equivalently to $(\mathbb{C}^*)^{\text{rk}(G)}$. Since $\mathfrak{g}_{\text{reg}}\cap\mathfrak{g}_{\text{ss}}$ is open and dense in $\mathfrak{g}_{\text{reg}}$, it follows that generic fibres of $\Phi$ are isomorphic to $(\mathbb{C}^*)^{\text{rk}(G)}$. This is not true of all fibres, however. To see this, suppose now that $x\in\mathcal{O}_{\text{reg}}$. It is known that $Z_G(x)$ decomposes as an internal direct product $Z(G)\times U_x$, where $Z(G)$ is the centre of $G$ and $U_x$ is a connected closed unipotent subgroup of $Z_G(x)$ (see \cite[Thm. 5.9(b)]{SpringerArithmetical}). The centre is finite by virtue of our having taken $G$ to be semisimple, so that $\dim(U_x)=\dim(Z_G(x))=\text{rk}(G)$. Also, as a connected unipotent group, $U_x$ is necessarily isomorphic to its Lie algebra (see \cite[Chapt. VIII, Thm. 1.1]{Hochschild}). In particular, $U_x\cong\mathbb{C}^{\text{rk}(G)}$ as varieties and it follows that $\Phi^{-1}(x)\cong Z_G(x)\cong Z(G)\times\mathbb{C}^{\text{rk}(G)}$ has $\vert Z(G)\vert$ connected components, each isomorphic to $\mathbb{C}^{\text{rk}(G)}$.   
	
	\subsection{The abstract integrable system}\label{Subsection: The abstract integrable system}
	
	While \cite{Fernandes} discusses abstract integrable systems in considerable generality, we shall focus on systems arising in a specific way. To this end, we will need to review a few definitions involving a holomorphic symplectic manifold $X$ and a holomorphic foliation $\mathcal{F}$ of $X$. Firstly, $\mathcal{F}$ is called an \textit{isotropic foliation} if its leaves are isotropic submanifolds of $X$. Secondly, one calls $\mathcal{F}$ \textit{Poisson complete} if the Poisson bracket of two locally defined first integrals of $\mathcal{F}$ is always itself a first integral. We may now state a holomorphic counterpart of Proposition 2.18(2) from \cite{Fernandes}. 
	
	\begin{proposition}\label{Proposition: Conditions for ANCI system}
		Let $X$ be a holomorphic symplectic manifold with a holomorphic foliation $\mathcal{F}$. If $\mathcal{F}$ is isotropic and Poisson complete, then $(X,\mathcal{F})$ is an abstract integrable system.
	\end{proposition}
	
	\begin{remark}
		A true holomorphic counterpart of \cite[Prop. 2.18(2)]{Fernandes}, as stated, would be slightly more general than what appears above. It would relax the requirement that $X$ be holomorphic symplectic, instead taking $X$ to be a holomorphic Poisson manifold having a \textit{regular} Poisson structure. For further details, we refer the reader to \cite{Fernandes}.
	\end{remark}
	
	Proposition \ref{Proposition: Conditions for ANCI system} will be our main technical tool for realizing an abstract integrable system on $G\times S_{\text{reg}}$, which we now discuss. Indeed, recall from Proposition \ref{Proposition: Surjective submersion} that $\Phi:G\times S_{\text{reg}}\rightarrow\mathfrak{g}$ is a holomorphic submersion. It follows that the connected components of $\Phi$'s fibres are the leaves of a holomorphic foliation $\mathcal{F}_{\text{reg}}$ of $G\times S_{\text{reg}}$.
	
	\begin{theorem}\label{Theorem: Phi is ANCI}
		With $\mathcal{F}_{\emph{reg}}$ as defined above, $(G\times S_{\emph{reg}},\mathcal{F}_{\emph{reg}})$ is an abstract integrable system of rank equal to $\emph{rk}(G)$.
	\end{theorem}
	
	\begin{proof}
		Corollary \ref{Corollary: Connected components} implies that $\mathcal{F}_{\text{reg}}$ is a $\text{rk}(G)$-dimensional foliation. It then just remains to prove that $(G\times S_{\text{reg}},\mathcal{F}_{\text{reg}})$ is an abstract integrable system, which we will accomplish by showing the hypotheses of Proposition \ref{Proposition: Conditions for ANCI system} to be satisfied. To begin, Proposition \ref{Proposition: Isotropic fibres} implies that $\mathcal{F}_{\text{reg}}$ is an isotropic foliation. Also, as $\Phi$ is a Poisson submersion (see Propositions \ref{Proposition: Surjective submersion} and \ref{Proposition: Poisson morphism}), Example 2.14 from \cite{Fernandes} explains that $\mathcal{F}_{\text{reg}}$ is necessarily Poisson complete. This concludes the proof.     
	\end{proof}
	
	\subsection{Moment maps and abstract integrable systems}\label{Subsection: Moment maps and abstract integrable systems}
	We now discuss a generalization of Theorem \ref{Theorem: Phi is ANCI}. To this end, recall that $\Phi:G\times S_{\text{reg}}\rightarrow\mathfrak{g}$ induces the abstract integrable system $(G\times S_{\text{reg}},\mathcal{F}_{\text{reg}})$ as follows: the leaves of $\mathcal{F}_{\text{reg}}$ are the connected components of $\Phi$'s fibres. We also know $\Phi$ to be a ($\mathfrak{g}$-valued) moment map (by Proposition \ref{Proposition: Hamiltonian action}), and it is natural to imagine that there are some general conditions under which a moment map will, analogously to $\Phi$, induce an abstract integrable system. This is indeed the case, as we shall establish. We will work in the holomorphic category for the sake of consistency with the rest of the paper, and the reader should interpret all relevant notions accordingly (ex. manifolds as complex manifolds, maps as holomorphic maps, etc.). Nevertheless, many parts of our discussion will also hold in the smooth category.    
	
	Let all notation be as established in Section \ref{Section: Background}, and let $X$ be a holomorphic symplectic manifold. Suppose that $X$ carries a Hamiltonian action of $G$. Using the isomorphism \eqref{Equation: Isomorphism of representations} to identify $\mathfrak{g}^*$ with $\mathfrak{g}$, we will present the moment map as $\mu:X\rightarrow\mathfrak{g}$. Also, given $x\in X$, let $Z_G(x)\subseteq G$ and $Z_{\mathfrak{g}}(x)\subseteq\mathfrak{g}$ denote the $G$-stabilizer of $x$ and its Lie algebra, respectively. We shall assume that the $G$-action on $X$ is \textit{locally free}, meaning that $Z_{\mathfrak{g}}(x)=\{0\}$ for all $x\in X$. This is equivalent to $\mu$ being a submersion (see \cite[Prop. III.2.3]{Audin}), and we may define $\mathcal{F}_{\mu}$ to be the holomorphic foliation of $X$ whose leaves are the connected components of $\mu$'s fibres. With this in mind, our generalization of Theorem \ref{Theorem: Phi is ANCI} will take the following form: finding conditions on $X$ and $\mu$ under which the proof of Theorem \ref{Theorem: Phi is ANCI} will show $(X,\mathcal{F}_{\mu})$ to be an abstract integrable system after we replace $G\times S_{\text{reg}}$, $\Phi$, and $\mathcal{F}_{\text{reg}}$ with $X$, $\mu$, and $\mathcal{F}_{\mu}$, respectively. Referring to the proof of Theorem \ref{Theorem: Phi is ANCI}, one readily sees that there is only one possible issue --- whether $\mathcal{F}_{\mu}$ is an isotropic foliation, or equivalently, all fibres of $\mu$ are isotropic in $X$.   
	
	\begin{theorem}\label{Theorem: General symplectic approach} 
		Let $X$ be a holomorphic symplectic manifold on which $G$ acts locally freely and in a Hamiltonian fashion with moment map $\mu:X\rightarrow\mathfrak{g}$. Then, $\mathcal{F}_{\mu}$ is an isotropic foliation of $X$ if and only if $\mu(X)\subseteq\mathfrak{g}_{\emph{reg}}$ and $\dim(X)=\dim(G)+\emph{rk}(G)$. In this case, $(X,\mathcal{F}_{\mu})$ is an abstract integrable system of rank equal to $\emph{rk}(G)$.
	\end{theorem}
	
	\begin{proof}
		Let $\omega$ denote the holomorphic symplectic form on $X$ and $\omega\vert_{\mu^{-1}(\mu(x))}$ its restriction to the level set $\mu^{-1}(\mu(x))$, $x\in X$. It follows that $\mathcal{F}_{\mu}$ is an isotropic foliation if and only if 
		\begin{equation}\label{Equation: First equivalent condition}
		\omega\vert_{\mu^{-1}(\mu(x))}=0\text{ for all }x\in X.
		\end{equation}
		Now let $(\omega\vert_{\mu^{-1}(\mu(x))})_x$ denote the bilinear form on $T_x(\mu^{-1}(\mu(x)))$ obtained by evaluating $\omega\vert_{\mu^{-1}(\mu(x))}$ at $x$, noting that \eqref{Equation: First equivalent condition} holds if and only if 
		\begin{equation}\label{Equation: Second equivalent condition}
		(\omega\vert_{\mu^{-1}(\mu(x))})_x=0\text{ for all }x\in X.
		\end{equation}
		The kernel of $(\omega\vert_{\mu^{-1}(\mu(x))})_x$ is the tangent space to the $Z_G(\mu(x))$-orbit of $x$, i.e. $T_x(Z_G(\mu(x))\cdot x)$ (see \cite[Lemma III.2.11]{Audin}), so that \eqref{Equation: Second equivalent condition} holds if and only if \begin{equation}\label{Equation: Third equivalent condition} 
		T_x(Z_G(\mu(x))\cdot x)=T_x(\mu^{-1}(\mu(x)))\text{ for all }x\in X.
		\end{equation}
		As $T_x(Z_G(\mu(x))\cdot x)$ is a subspace of $T_x(\mu^{-1}(\mu(x)))$, \eqref{Equation: Third equivalent condition} is true if and only if \begin{equation}\label{Equation: Fourth equivalent condition}
		\dim(T_x(Z_G(\mu(x))\cdot x))=\dim(T_x(\mu^{-1}(\mu(x)))) \text{ for all }x\in X.
		\end{equation}
		In the interest of modifying \eqref{Equation: Fourth equivalent condition}, we make two observations. Firstly, $\mu$ being a submersion implies $\dim(T_x(\mu^{-1}(\mu(x))))=\dim(X)-\dim(\mathfrak{g})=\dim(X)-\dim(G)$. Secondly, since the $G$-action is locally free, we must have $\dim(T_x(Z_G(\mu(x))\cdot x))=\dim(Z_G(\mu(x)))$. It follows that \eqref{Equation: Fourth equivalent condition} holds if and only if 
		\begin{equation}\label{Equation: Fifth equivalent condition}
		\dim(Z_G(\mu(x)))=\dim(X)-\dim(G) \text{ for all }x\in X.
		\end{equation}
		
		By virtue of the discussion above, we are reduced to showing that \eqref{Equation: Fifth equivalent condition} holds if and only if $\mu(X)\subseteq\mathfrak{g}_{\text{reg}}$ and $\dim(X)=\dim(G)+\text{rk}(G)$. To this end, assume that \eqref{Equation: Fifth equivalent condition} is satisfied. Since $\mu$ is a submersion, its image $\mu(X)$ is necessarily open in $\mathfrak{g}$. The set of regular elements is dense in $\mathfrak{g}$ (as discussed in Section \ref{Subsection: Some general Lie-theoretic preliminaries}) and must therefore intersect $\mu(X)$, i.e. $\mu(y)\in\mathfrak{g}_{\text{reg}}$ for some $y\in X$. Note that $\dim(Z_G(\mu(y)))=\text{rk}(G)$, which together with \eqref{Equation: Fifth equivalent condition} gives $\dim(X)=\dim(G)+\text{rk}(G)$. Moreover, \eqref{Equation: Fifth equivalent condition} now reads as
		$$\dim(Z_G(\mu(x)))=\text{rk}(G) \text{ for all }x\in X.$$
		This is the statement that $\mu(x)\in\mathfrak{g}_{\text{reg}}$ for all $x\in X$, or equivalently $\mu(X)\subseteq\mathfrak{g}_{\text{reg}}$.
		
		Conversely, assume that $\mu(X)\subseteq\mathfrak{g}_{\text{reg}}$ and $\dim(X)=\dim(G)+\text{rk}(G)$. It is then immediate that both sides of \eqref{Equation: Fifth equivalent condition} coincide with $\text{rk}(G)$ for all $x\in X$, so that \eqref{Equation: Fifth equivalent condition} holds. This completes the proof.              
	\end{proof}
	
	\section{Some integrable systems}\label{Section: Some integrable systems}
	While the abstract integrable system $(G\times S_{\text{reg}},\mathcal{F}_{\text{reg}})$ has the virtue of being completely canonical, it lacks the explicit Hamiltonian functions of a traditional integrable system. Nevertheless, it is possible to construct integrable systems on $G\times S_{\text{reg}}$. We shall illustrate this in two ways, devoting Section \ref{Subsection: An integrable system of rank equal to rk(G)} to the first and Section \ref{Subsection: A family of completely integrable systems} to the second.
	
	\subsection{An integrable system of rank equal to $\text{rk}(G)$}\label{Subsection: An integrable system of rank equal to rk(G)} 
	Consider the algebra $\mathbb{C}[\mathfrak{g}]:=\mathrm{Sym}(\mathfrak{g}^*)$ of polynomial functions on the variety $\mathfrak{g}$. The Poisson structure on $\mathfrak{g}$ (discussed in Section \ref{Subsection: Some general Lie-theoretic preliminaries}) gives $\mathbb{C}[\mathfrak{g}]$ the structure of a Poisson algebra. Also, the adjoint action induces a representation of $G$ on $\mathbb{C}[\mathfrak{g}]$, and one can form the subalgebra $\mathbb{C}[\mathfrak{g}]^G\subseteq\mathbb{C}[\mathfrak{g}]$ of $G$-invariant polynomials. Each of these invariant polynomials Poisson-commutes with every polynomial on $\mathfrak{g}$, i.e. 
	\begin{equation}\label{Equation: Poisson centre}
	\{f,h\}=0\text{ for all }f\in\mathbb{C}[\mathfrak{g}]^G,\text{ } h\in\mathbb{C}[\mathfrak{g}].
	\end{equation}
	Also, it is a celebrated fact that $\mathbb{C}[\mathfrak{g}]^G$ is generated by $\text{rk}(G)$ algebraically independent homogeneous generators. Let $f_1,f_2,\ldots,f_{\text{rk}(G)}\in\mathbb{C}[\mathfrak{g}]^G$ be a choice of such generators, fixed for the rest of this paper, and consider the map
	\begin{eqnarray}F\;:\;\mathfrak{g} & \rightarrow & \mathbb{C}^{\text{rk}(G)}\nonumber\\ x & \mapsto & (f_1(x),f_2(x),\ldots,f_{\text{rk}(G)}(x)).\nonumber\end{eqnarray}
	It is known that $\mathfrak{g}_{\text{reg}}$ is the locus on which $df_1,df_2,\ldots,df_{\text{rk}(G)}$ are linearly independent, or equivalently
	\begin{equation}\label{Equation: Functional independence}
	\mathfrak{g}_{\text{reg}}=\{x\in\mathfrak{g}:\text{rank}(d_xF)=\text{rk}(G)\}
	\end{equation}   
	(see \cite[Thm. 9]{KostantLie}).
	
	Now choose a basis $\{\theta_1,\theta_2,\ldots,\theta_{\dim(G)}\}$ of $\mathfrak{g}^*$, viewed as a system of global holomorphic coordinates on $\mathfrak{g}$. For each $x\in\mathfrak{g}$, let $[d_xF]$ denote the $\text{rk}(G)\times\dim(G)$ Jacobian matrix representative of $d_xF$, i.e. 
	\begin{equation}\label{Equation: First Jacobian}[d_xF]_{ij}:=\frac{\partial f_i}{\partial\theta_j}(x),\quad i=1,\ldots,\text{rk}(G), \text{ }j=1,\ldots,\dim(G).
	\end{equation}
	Choosing a point $y\in\mathfrak{g}_{\text{reg}}$, we may use \eqref{Equation: Functional independence} and conclude that $[d_yF]$ has rank equal to $\text{rk}(G)$. It follows that 
	\begin{equation}\label{Equation: Nonvanishing minor}
	\det(M(y))\neq 0
	\end{equation}
	some $\text{rk}(G)\times\text{rk}(G)$ submatrix $M(y)$ of $[d_yF]$. Reordering our basis vectors if necessary, we may assume that this submatrix consists of the first $\text{rk}(G)$ columns. Now consider the holomorphic functions $\Phi_1,\Phi_2,\ldots,\Phi_{\dim(G)}:G\times S_{\text{reg}}\rightarrow\mathbb{C}$ defined as follows:
	\begin{equation}\label{Equation: Holomorphic functions for NCI system}
	\Phi_i=\begin{cases}f_i\circ\Phi &\mbox{if } i=1,\ldots,\text{rk}(G) \\ 
	\theta_i\circ\Phi & \mbox{if } i=\text{rk}(G)+1,\ldots,\dim(G).\end{cases}.
	\end{equation}
	
	\begin{theorem}\label{Theorem: An integrable system}
		The functions $\Phi_i$ form an integrable system on $G\times S_{\emph{reg}}$, and the rank of this system is $\emph{rk}(G)$.
	\end{theorem}
	
	\begin{proof}
		Using \eqref{Equation: Poisson centre}, one sees that 
		$$\{f_i,h\}=0\text{ for all } i\in\{1,\ldots,\text{rk}(G)\}\text{ and }h\in\mathbb{C}[\mathfrak{g}].$$ 
		Since $\Phi$ is a Poisson morphism (by Proposition \ref{Proposition: Poisson morphism}), it follows that
		$$\{f_i\circ\Phi,h\circ\Phi\}=0\text{ for all } i\in\{1,\ldots,\text{rk}(G)\}\text{ and }h\in\mathbb{C}[\mathfrak{g}].$$
		Replacing $f_i\circ\Phi$ with $\Phi_i$ and choosing $h$ appropriately, we obtain 
		$$\{\Phi_i,\Phi_j\}=0\text{ for all }i\in\{1,\ldots,\text{rk}(G)\},\text{ }j\in\{1,\ldots,\dim(G)\}.$$
		
		It remains only to prove that $d\Phi_1,d\Phi_2,\ldots,d\Phi_{\dim(G)}$ are linearly independent on an open dense subset of $G\times S_{\text{reg}}$. To this end, consider the holomorphic map
		\begin{eqnarray}H\;:\;\mathfrak{g} & \rightarrow & \mathbb{C}^{\dim(G)}\nonumber\\ x & \mapsto & (f_1(x),\ldots,f_{\text{rk}(G)}(x),\theta_{\text{rk}(G)+1}(x),\ldots,\theta_{\dim(G)}(x)),\nonumber\end{eqnarray} and note that
		$$(H\circ\Phi)(g,x)=(\Phi_1(g,x),\Phi_2(g,x),\ldots,\Phi_{\dim(G)}(g,x)),\quad (g,x)\in G\times S_{\text{reg}}.$$
		It follows that the linear independence of $d\Phi_1,d\Phi_2,\ldots,d\Phi_{\dim(G)}$ at a point is equivalent to $d(H\circ\Phi)$ having full rank at the same point. Also, as $\Phi$ is a submersion (by Proposition \ref{Proposition: Surjective submersion}), $d(H\circ\Phi)$ has full rank at $(g,x)\in G\times S_{\text{reg}}$ if and only if $dH$ has full rank at $\Phi(g,x)$. By virtue of these last two sentences, it will suffice to prove that the open set
		\begin{equation}\label{Equation: First open set}
		U:=\Phi^{-1}\big(\{x\in\mathfrak{g}:\text{rank}(d_xH)=\dim(G)\}\big)\subseteq G\times S_{\text{reg}}
		\end{equation}
		is dense. Accordingly, recall the Jacobian matrix construction \eqref{Equation: First Jacobian}. One analogously has a Jacobian matrix representative $[d_xH]$ for each linear map $d_xH$, $x\in\mathfrak{g}$. It is not difficult to see that $[d_xH]$ is block upper-triangular with two diagonal blocks, one consisting of the first $\text{rk}(G)$ columns of $[d_xF]$ and the other an identity matrix. The former block, to be denoted $M(x)$, must therefore have determinant equal to that of $[d_xH]$. We shall let $\rho(x)$ denote this common determinant, i.e.
		$$\rho(x):=\det(M(x))=\det([d_xH]),\quad x\in\mathfrak{g}.$$
		Now note that $\text{rank}(d_xH)=\dim(G)$ if and only if $\rho(x)\neq0$, so that \eqref{Equation: First open set} becomes \begin{equation}\label{Equation: Second open set}
		U=\Phi^{-1}(\mathfrak{g}\setminus\rho^{-1}(0))=(G\times S_{\text{reg}})\setminus (\rho\circ\Phi)^{-1}(0).
		\end{equation} 
		Moreover, \eqref{Equation: Nonvanishing minor} implies that $\rho(y)\neq 0$ for some $y\in\mathfrak{g}_{\text{reg}}$. Since $\mathfrak{g}_{\text{reg}}$ is the image of $\Phi$ (by Proposition \ref{Proposition: Surjective submersion}), we can write $y=\Phi(g,x)$ for some $(g,x)\in G\times S_{\text{reg}}$. Note that the condition $\rho(y)\neq 0$ then becomes $(\rho\circ\Phi)(g,x)\neq 0$, meaning that $\rho\circ\Phi$ is not identically zero. As a holomorphic function with this property, the complement of its vanishing locus is necessarily dense in $G\times S_{\text{reg}}$. This complement is precisely $U$ by \eqref{Equation: Second open set}, completing the proof. 
	\end{proof}
	
	\subsection{A family of completely integrable systems}\label{Subsection: A family of completely integrable systems}
	While Theorem \ref{Theorem: An integrable system} provides an integrable system, the system itself is not completely integrable. Indeed, the dimension of $G\times S_{\text{reg}}$ (equal to $\dim(G)+\text{rk}(G)$) is more than twice the rank of this integrable system (equal to $\text{rk}(G)$, by Theorem \ref{Theorem: An integrable system}). In what follows, however, we will show that $G\times S_{\text{reg}}$ carries a family of completely integrable systems parametrized by the regular semisimple elements $\mathfrak{g}_{\text{reg}}\cap\mathfrak{g}_{\text{ss}}$. Our arguments will make extensive use of results on maximal Poisson-commutative subalgebras of polynomial algebras, developed by Mishchenko and Fomenko in \cite{MishchenkoEuler} and summarized by Rybnikov in \cite{Rybnikov}. In more detail, we may associate to each $\beta\in\mathfrak{g}_{\text{reg}}\cap\mathfrak{g}_{\text{ss}}$ the following family of polynomials:
	\begin{equation}\label{Equation: Family of polynomials}f_{ij}^{\beta}:=(\partial_{\beta})^j(f_i)\in\mathbb{C}[\mathfrak{g}],\quad i=1,\ldots,\text{rk}(G),\quad j=0,\ldots,\deg(f_i),\end{equation}
	where $\partial_{\beta}$ is the operator for taking a directional derivative in the direction $\beta$. The following is one of Mishchenko and Fomenko's results, as presented in Section 2 of \cite{Rybnikov}.
	
	\begin{theorem}[Mishchenko-Fomenko]\label{Theorem: Mishchenko and Fomenko's theorem}
		If $\beta\in\mathfrak{g}_{\emph{reg}}\cap\mathfrak{g}_{\emph{ss}}$, then \eqref{Equation: Family of polynomials} is a list of $\frac{1}{2}(\dim(G)+\emph{rk}(G))$ algebraically independent polynomials on $\mathfrak{g}$, and these polynomials generate a maximal Poisson-commutative subalgebra of $\mathbb{C}[\mathfrak{g}]$. 	
	\end{theorem}
	
	Now fix $\beta\in\mathfrak{g}_{\text{reg}}\cap\mathfrak{g}_{\text{ss}}$ and consider the holomorphic functions on $G\times S_{\text{reg}}$ obtained by pulling back those in \eqref{Equation: Family of polynomials} along $\Phi$, i.e.
	\begin{equation}\label{Equation: New family of functions}\Phi_{ij}^{\beta}:=f_{ij}^{\beta}\circ\Phi, \quad i=1,\ldots,\text{rk}(G),\quad j=0,\ldots,\deg(f_i).\end{equation}
	
	\begin{theorem}\label{Theorem: Family of completely integrable systems}
		If $\beta\in\mathfrak{g}_{\emph{reg}}\cap\mathfrak{g}_{\emph{ss}}$, then the functions in \eqref{Equation: New family of functions} form a completely integrable system on $G\times S_{\emph{reg}}$.
	\end{theorem}
	
	\begin{proof}
		We begin by noting that \eqref{Equation: New family of functions} is a list of $\frac{1}{2}(\dim(G)+\text{rk}(G))$ functions (see Theorem \ref{Theorem: Mishchenko and Fomenko's theorem}). This number is precisely half the dimension of $G\times S_{\text{reg}}$, so we need only prove that the $\Phi_{ij}^{\beta}$ Poisson commute in pairs and have linearly independent differentials on an open dense subset of $G\times S_{\text{reg}}$. To establish the first of these properties, recall from Proposition \ref{Proposition: Poisson morphism} that $\Phi$ is a Poisson morphism. In particular, if two functions on $\mathfrak{g}$ Poisson commute, their respective pullbacks under $\Phi$ must also Poisson commute. Since the $f_{ij}^{\beta}$ Poisson commute in pairs (by Theorem \ref{Theorem: Mishchenko and Fomenko's theorem}), it follows that the same must be true of the pullbacks $\Phi^*(f_{ij}^{\beta})=\Phi_{ij}^{\beta}$.
		
		To see that the $\Phi_{ij}^{\beta}$ have linearly independent differentials on an open dense subset of $G\times S_{\text{reg}}$, set $p:=\frac{1}{2}(\dim(G)+\text{rk}(G))$, re-index the $f_{ij}^{\beta}$ as $f_1^{\beta},f_2^{\beta},\ldots,f_p^{\beta}$, and consider the holomorphic map
		\begin{eqnarray}H\;:\;\mathfrak{g} & \rightarrow & \mathbb{C}^{p}\nonumber\\ x & \mapsto & (f_1^{\beta}(x),f_2^{\beta}(x),\ldots,f_{p}^{\beta}(x)).\nonumber\end{eqnarray}
		By arguments analogous to those appearing in the proof of Theorem \ref{Theorem: An integrable system}, we need only show that the open subset $$
		U:=\Phi^{-1}\big(\{x\in\mathfrak{g}:\text{rank}(d_xH)=p\}\big)\subseteq G\times S_{\text{reg}}$$ is dense. To this end, recall the global holomorphic coordinates on $\mathfrak{g}$ fixed in Section \ref{Subsection: An integrable system of rank equal to rk(G)}, and let $[d_xH]$ denote the resulting Jacobian matrix representative of $d_xH$. Using the fact that the $f_{ij}^{\beta}$ are algebraically independent (see Theorem \ref{Theorem: Mishchenko and Fomenko's theorem}), it is straightforward to find a $p\times p$ submatrix $M(x)$ of $[d_xH]$ whose determinant is not identically zero as a function of $x$ (see \cite[Section 9]{Ito}, for instance). In other words, the function
		\begin{eqnarray}\rho\;:\;\mathfrak{g} & \rightarrow & \mathbb{C}\nonumber\\ x & \mapsto & \det(M(x))\nonumber\end{eqnarray} is not identically zero. Now note that $\text{rank}(d_xH)=p$ whenever $\rho(x)\neq 0$, or equivalently
		$$\mathfrak{g}\setminus\rho^{-1}(0)\subseteq \{x\in\mathfrak{g}:\text{rank}(d_xH)=p\}.$$ Taking preimages under $\Phi$, we obtain
		$$(G\times S_{\text{reg}})\setminus(\rho\circ\Phi)^{-1}(0)\subseteq U.$$ We are therefore reduced to showing that $(G\times S_{\text{reg}})\setminus(\rho\circ\Phi)^{-1}(0)$ is dense in $G\times S_{\text{reg}}$, for which our argument is similar to one given in the proof of Theorem \ref{Theorem: An integrable system}. Specifically, as $\rho$ is not identically zero, the complement of its vanishing locus is nonempty and must intersect the dense subset $\mathfrak{g}_{\text{reg}}\subseteq\mathfrak{g}$. Since $\mathfrak{g}_{\text{reg}}$ is the image of $\Phi$ (by Proposition \ref{Proposition: Surjective submersion}), it follows that $\rho\circ\Phi$ is not identically zero. As argued in the proof of Theorem \ref{Theorem: An integrable system}, this shows that $(G\times S_{\text{reg}})\setminus (\rho\circ\Phi)^{-1}(0)$ is dense in $G\times S_{\text{reg}}$.
	\end{proof}   	  
	\bibliographystyle{acm} 
	\bibliography{AbsIntSys}

\begin{thebibliography}{10}

\bibitem{AdamsHarnadHurtubise}
{\sc Adams, M.~R., Harnad, J., and Hurtubise, J.}
\newblock Darboux coordinates on coadjoint orbits of {L}ie algebras.
\newblock {\em Lett. Math. Phys. 40}, 1 (1997), 41--57.

\bibitem{Audin}
{\sc Audin, M.}
\newblock {\em Torus actions on symplectic manifolds}, revised~ed., vol.~93 of
  {\em Progress in Mathematics}.
\newblock Birkh\"auser Verlag, Basel, 2004.

\bibitem{Beauville}
{\sc Beauville, A.}
\newblock Jacobiennes des courbes spectrales et syst\`emes hamiltoniens
  compl\`etement int\'egrables.
\newblock {\em Acta Math. 164}, 3-4 (1990), 211--235.

\bibitem{Bielawski}
{\sc Bielawski, R.}
\newblock Hyperk\"{a}hler structures and group actions.
\newblock {\em J. London Math. Soc. (2) 55}, 2 (1997), 400--414.

\bibitem{BiswasRamanan}
{\sc Biswas, I., and Ramanan, S.}
\newblock An infinitesimal study of the moduli of {H}itchin pairs.
\newblock {\em J. London Math. Soc. (2) 49}, 2 (1994), 219--231.

\bibitem{Bolsinov}
{\sc Bolsinov, A.~V., and Jovanovi\'c, B.~z.}
\newblock Noncommutative integrability, moment map and geodesic flows.
\newblock {\em Ann. Global Anal. Geom. 23}, 4 (2003), 305--322.

\bibitem{Bottacin}
{\sc Bottacin, F.}
\newblock Symplectic geometry on moduli spaces of stable pairs.
\newblock {\em Ann. Sci. \'Ecole Norm. Sup. (4) 28}, 4 (1995), 391--433.

\bibitem{Chriss}
{\sc Chriss, N., and Ginzburg, V.}
\newblock {\em Representation theory and complex geometry}.
\newblock Modern Birkh\"auser Classics. Birkh\"auser Boston, Inc., Boston, MA,
  2010.
\newblock Reprint of the 1997 edition.

\bibitem{Collingwood}
{\sc Collingwood, D.~H., and McGovern, W.~M.}
\newblock {\em Nilpotent orbits in semisimple {L}ie algebras}.
\newblock Van Nostrand Reinhold Mathematics Series. Van Nostrand Reinhold Co.,
  New York, 1993.

\bibitem{DonagiMarkman}
{\sc Donagi, R., and Markman, E.}
\newblock Spectral covers, algebraically completely integrable, {H}amiltonian
  systems, and moduli of bundles.
\newblock In {\em Integrable systems and quantum groups ({M}ontecatini {T}erme,
  1993)}, vol.~1620 of {\em Lecture Notes in Math.} Springer, Berlin, 1996,
  pp.~1--119.

\bibitem{Fernandes}
{\sc Fernandes, R.~L., Laurent-Gengoux, C., and Vanhaecke, P.}
\newblock Global action-angle variables for non-commutative integrable systems.
  arxiv:1503.00084 (2015), 44 pages. {T}o appear in {J}. {S}ymplectic {G}eom.

\bibitem{Fiorani}
{\sc Fiorani, E., and Sardanashvily, G.}
\newblock Noncommutative integrability on noncompact invariant manifolds.
\newblock {\em J. Phys. A 39}, 45 (2006), 14035--14042.

\bibitem{HauselThaddeus}
{\sc Hausel, T., and Thaddeus, M.}
\newblock Mirror symmetry, {L}anglands duality, and the {H}itchin system.
\newblock {\em Invent. Math. 153}, 1 (2003), 197--229.

\bibitem{Hitchin}
{\sc Hitchin, N.}
\newblock Stable bundles and integrable systems.
\newblock {\em Duke Math. J. 54}, 1 (1987), 91--114.

\bibitem{Hochschild}
{\sc Hochschild, G.~P.}
\newblock {\em Basic theory of algebraic groups and {L}ie algebras}, vol.~75 of
  {\em Graduate Texts in Mathematics}.
\newblock Springer-Verlag, New York-Berlin, 1981.

\bibitem{Ito}
{\sc Ito, H.}
\newblock Convergence of {B}irkhoff normal forms for integrable systems.
\newblock {\em Comment. Math. Helv. 64}, 3 (1989), 412--461.

\bibitem{MirandaKiesenhofer}
{\sc Kiesenhofer, A., and Miranda, E.}
\newblock Noncommutative integrable systems on {$b$}-symplectic manifolds.
\newblock {\em Regul. Chaotic Dyn. 21}, 6 (2016), 643--659.

\bibitem{Kirillov}
{\sc Kirillov, Jr., A.}
\newblock {\em An introduction to {L}ie groups and {L}ie algebras}, vol.~113 of
  {\em Cambridge Studies in Advanced Mathematics}.
\newblock Cambridge University Press, Cambridge, 2008.

\bibitem{KostantLie}
{\sc Kostant, B.}
\newblock Lie group representations on polynomial rings.
\newblock {\em Amer. J. Math. 85\/} (1963), 327--404.

\bibitem{MirandaLaurantVanhaecke}
{\sc Laurent-Gengoux, C., Miranda, E., and Vanhaecke, P.}
\newblock Action-angle coordinates for integrable systems on {P}oisson
  manifolds.
\newblock {\em Int. Math. Res. Not. IMRN}, 8 (2011), 1839--1869.

\bibitem{Markman}
{\sc Markman, E.}
\newblock Spectral curves and integrable systems.
\newblock {\em Compositio Math. 93}, 3 (1994), 255--290.

\bibitem{Marsden}
{\sc Marsden, J.~E., Ratiu, T., and Raugel, G.}
\newblock Symplectic connections and the linearisation of {H}amiltonian
  systems.
\newblock {\em Proc. Roy. Soc. Edinburgh Sect. A 117}, 3-4 (1991), 329--380.

\bibitem{MishchenkoEuler}
{\sc Mishchenko, A.~S., and Fomenko, A.~T.}
\newblock Euler equation on finite-dimensional {L}ie groups.
\newblock {\em Izv. Akad. Nauk SSSR Ser. Mat. 42}, 2 (1978), 396--415, 471.

\bibitem{Mishchenko}
{\sc Mishchenko, A.~S., and Fomenko, A.~T.}
\newblock A generalized {L}iouville method for the integration of {H}amiltonian
  systems.
\newblock {\em Funkcional. Anal. i Prilo\v zen. 12}, 2 (1978), 46--56, 96.

\bibitem{Ortega}
{\sc Ortega, J.-P., and Ratiu, T.~S.}
\newblock {\em Momentum maps and {H}amiltonian reduction}, vol.~222 of {\em
  Progress in Mathematics}.
\newblock Birkh\"auser Boston, Inc., Boston, MA, 2004.

\bibitem{Reshetikhin}
{\sc Reshetikhin, N.}
\newblock Degenerately integrable systems.
\newblock {\em Zap. Nauchn. Sem. S.-Peterburg. Otdel. Mat. Inst. Steklov.
  (POMI) 433}, Voprosy Kvantovo\u\i \ Teorii Polya i Statistichesko\u\i \
  Fiziki. 23 (2015), 224--245.

\bibitem{Rybnikov}
{\sc Rybnikov, L.~G.}
\newblock The argument shift method and the {G}audin model.
\newblock {\em Funktsional. Anal. i Prilozhen. 40}, 3 (2006), 30--43, 96.

\bibitem{Slodowy}
{\sc Slodowy, P.}
\newblock {\em Simple singularities and simple algebraic groups}, vol.~815 of
  {\em Lecture Notes in Mathematics}.
\newblock Springer, Berlin, 1980.

\bibitem{SpringerArithmetical}
{\sc Springer, T.~A.}
\newblock Some arithmetical results on semi-simple {L}ie algebras.
\newblock {\em Inst. Hautes \'Etudes Sci. Publ. Math.}, 30 (1966), 115--141.

\bibitem{Vanhaecke}
{\sc Vanhaecke, P.}
\newblock Algebraic integrability: a survey.
\newblock {\em Philos. Trans. R. Soc. Lond. Ser. A Math. Phys. Eng. Sci. 366},
  1867 (2008), 1203--1224.

\end{thebibliography}
\end{document}